\documentclass[10 pt, a4 paper, oneside, openany]{article}
\usepackage[T1]{fontenc}
\usepackage{amsmath}
\usepackage[utf8]{inputenc}
\usepackage[english]{babel}
\usepackage{graphicx}
\usepackage{enumerate}
\usepackage{enumitem}
\usepackage{amsthm}
\usepackage{amsfonts}
\usepackage{amssymb}
\usepackage{amsthm}
\usepackage{wrapfig}
\usepackage{pgfplots}
\pgfplotsset{compat=1.15}
\usepackage{mathrsfs}
\usepackage{hyperref}
\usetikzlibrary{arrows}
\usepackage[swapnames, nouppercase]{frontespizio} 
\usepackage{geometry}
\usepackage{tikz}
\usepackage{float}
\usepackage{appendix}
\usepackage{setspace}
\allowdisplaybreaks

\newgeometry{right=1.5cm} 

\newtheorem{teorema}{Theorem}[section]

\newtheorem{lemma}[teorema]{Lemma}

\theoremstyle{definition}
\newtheorem{osservazione}[teorema]{Remark}
\newtheorem{definition}[teorema]{Definition}

\begin{document}
\title{ Finite Index and Do Carmo’s Question\\ for Constant Mean Curvature Hypersurfaces}
\date{}
\author{Barbara Nelli, Claudia Pontuale}

\maketitle

\begin{abstract}
We prove that any  finite $\delta$-index hypersurface $M$ in ${\mathbb R}^{n+1}$  with constant mean curvature must be minimal, provided
either of the following conditions holds: \begin{itemize}
\item the volume growth of $M$ is sub-exponential;
\item the Ricci curvature of $M$ satisfies $\operatorname{Ric}_M\geq -\frac{C(1-\delta)}{n-1}|A|^2g,$ where  $A$ is the second fundamental form, $g$ is the metric on $M$ and $C$ is a positive constant smaller than 4. 
\end{itemize}
We emphasize that no restriction on the dimension is imposed. Moreover, the  statement in the second case is new even for finite index hypersurfaces ($\delta=0$). 
\end{abstract}

\let\thefootnote\relax\footnote{The authors were partially supported by PRIN-2022AP8HZ9 and  INdAM-GNSAGA.}



\section*{Introduction}

In \cite{DC} in 1989 Manfredo  Perdig\~{a}o Do Carmo asked: 

\

{\em 
Is a noncompact, complete, stable, constant mean curvature  hypersurface  of ${\mathbb R}^{n+1}$ necessarily minimal?}

\
 
 The question was motivated by the fact that the answer is positive in ${\mathbb R}^3$ \cite{LR, Si}. Moreover, in the case of ${\mathbb R}^3,$  the generalization of Bernstein theorem to stable minimal surfaces \cite{DoPe, FCSc, Po} yields that the surface is finally a plane.

Do Carmo's question has been investigated  in ${\mathbb R}^{n+1}$ and in more general ambient manifolds and in the broader class of finite index constant mean curvature  hypersurfaces.
 A positive answer for $n=3, 4$ was obtained independently by M. F. Elbert, the first author and H. Rosenberg \cite{elbertnelliros} and X. Cheng \cite{cheng}, using the {\em Bonnet-Myers method} nowadays also known as the {\em conformal method}.
 More precisely, there is no stable hypersurface with  constant mean curvature $H$ ($H$-hypersurface) in a manifold  $N$ of dimension $n+1,$ $n=3,4$ provided $H$ is large enough with respect to the sectional curvature of the ambient manifold $N$. In particular, in the Euclidean space this implies nonexistence for any $H\not=0$. 
 Moreover, partial results for a hypersurface $M$ of arbitrary dimension have been obtained by several authors (see for example  \cite{alencardocarmo} for $M$ with polynomial volume growth, \cite{entropies} for $M$ with subexponential volume growth  and \cite{AlCa, CaZh} for $M$  satisfying  assumptions on the total curvature, and  $n\leq 5$).
 
Recently, advances have been made in low dimension. Using the $\mu$-bubble technique introduced by O. Chodosh, C. Li, P. Minter, D. Stryker in \cite{chodoshlimintstryker} (see also \cite{mazet}), J. Chen, H. Hong, H. Li \cite{chenhongli} answerd positively to Do Carmo question in $\mathbb R^6$. 

In this context,  it is natural to ask Do Carmo’s question in the broader framework of $\delta$-stable hypersurfaces (see Definition \ref{def-deltastable}). 

Tinaglia and Zhou \cite{tinagliazhou} answered affirmatively to Do Carmo's question for $\delta$- stable H-hypersurfaces in $\mathbb R^{n+1},$ $n=2,3,4$  for certain values of $\delta,$ using  the conformal method. In their recent paper \cite{chenhongli}, J. Chen, H. Hong, H. Li  also noted that their approach using the $\mu$-bubble method to the classical Do Carmo problem applies to the $\delta$-stable case in $\mathbb{R}^6$, again for some values of $\delta$.

We notice that, most of the results we cited,  hold for the more general class of finite $\delta$-index hypersurfaces (see Definition \ref{def-deltaindex}).

Our interest here is twofold. On the one hand, in the context of Do Carmo’s problem, we generalize to the finite 
$\delta$-index case some results available in arbitrary dimension for hypersurfaces with finite index. On the other hand, we look for new geometric conditions under which Do Carmo’s question has a positive answer in any dimension within the  class of finite 
$\delta$-index hypersurfaces.

Our main results concerning Do Carmo's question in the Euclidean space  are the following.

\begin{teorema}  
\label{docarmo-entropy-intro}
   There are no complete noncompact finite $\delta$-index   $H$-hypersurfaces with $H\not=0$ and subexponential volume growth, immersed in ${\mathbb R}^{n+1}$.\end{teorema}

 We remark that the condition of having subexponential volume growth is weaker than any polynomial bound on the volume growth and that the analogous result for finite index hypersurfaces ($\delta=0$) was proved in \cite{entropies}.

\

In addressing do Carmo’s question, the next theorem introduces a hypothesis that, to the best of our knowledge, has not previously been considered. Nevertheless, this perspective will be shown to be quite natural (see Chapter 3).

\begin{teorema}
\label{docarmo-riccibounded-intro}
  A complete, noncompact, finite $\delta$-index, $H$-hypersurface $M$ with $$\operatorname{Ric}_M \ge- \frac{C(1-\delta)}{n-1}|A|^2g,$$ immersed ${\mathbb R}^{n+1}$, where $C$ is a positive constant smaller than 4,  is minimal.
 \end{teorema}

Notice that,  Theorem  \ref{docarmo-riccibounded-intro} is new also in the case of stable hypersurfaces ($\delta=0$).  Moreover one can adapt  the proof of Theorem \ref{docarmo-riccibounded-intro} to the case of an ambient manifold $N$  with curvature tensor bounded below obtaining that there is no hypersurface with  constant mean curvature  $H$ large enough (see Theorem \ref{docarmo-riccibounded-general}).

\

The paper is organized as follows. In Section 1, we introduce the necessary notions and recall some basic properties. In Section 2, we address Do Carmo’s question in the case of subexponential volume growth. In Section 3,  we prove Theorem \ref{docarmo-riccibounded-intro} and its extension to more general ambient spaces. Finally, in Appendix A we establish some useful Caccioppoli-type inequalities for $\delta$-stable hypersurfaces.

\section{Definitions and Basic Properties}
\label{basics}

Throughout this article $M$ will be a complete, noncompact $n$-dimensional hypersurface with constant mean curvature $H$ ($H$-hypersurface) of a Riemannian manifold $N$, $\nu$ its unit normal vector and $A$ its second fundamental form. 

We start by recalling the notion of $\delta$-stability and $\delta$-index that are natural generalizations of the classical concepts of stability and index. 

We mention here some references about $\delta$-stable hypersurfaces:
\cite{ChZh, ChLi, Fu, HoLiWa, MePe, tinagliazhou}.

\begin{definition}
\label{def-deltastable}
For $\delta \in [0,1)$, we say that an $H$-hypersurface $M$ immersed in $N$ is \textit{$\delta$-stable} if for all $f \in C_0^\infty(M)$,
\[
Q_\delta (f, f)=\int_M |\nabla f|^2 - (1 - \delta)(|A|^2 + \operatorname{Ric}_N(\nu, \nu)) f^2 \geq 0.
\]
When $\delta = 0$, then $M$ is (classically) stable.
\end{definition}


The {\em $\delta$-stability operator} is defined as
\begin{equation*}
    J_\delta(f)= \Delta f + (1- \delta) (|A|^2+ \operatorname{Ric}_N(\nu, \nu))f.
\end{equation*}

\begin{definition}
\label{def-deltaindex}
The $\delta$-\textit{index} of a compact subset \( K \subset M \) is defined to be the number  ${i_\delta}$ of negative eigenvalues (counted with multiplicity) of the  operator \( -J_\delta \).
\end{definition}
\noindent If \( K_1 \subset K_2 \), then
\[
\operatorname{i_\delta}(K_1 ) \leq \operatorname{i_\delta}(K_2),
\]
hence, we may define the $\delta$-\textit{index} of \( M \) as
\[
\operatorname{i_\delta}(M) := \lim_{R \to \infty} \operatorname{i_\delta}(B_{\sigma}^M(R)),
\]
where \( B_{\sigma}^M(R) \) is the geodesic ball of radius \( R \) centered at a fixed point $\sigma$ of the manifold $M$.

Let \( w \) be a positive non-decreasing function. The number 
\begin{equation*}
\mu_w := \limsup_{r \to \infty} \frac{\ln w(r)}{r}
\end{equation*}
 
 is called the {\em  logarithmic growth} of $w.$

We are  mainly interested in  the logarithmic growth of the volume  of a manifold. Denote by \( |B^M_\sigma(R)| \) the  volume of the geodesic ball \( B^M_\sigma(R) \).

\begin{definition}
The \textit{logarithmic volume growth} of \( M \)  is defined as
\begin{equation*}
\mu_M := \limsup_{R \to \infty} \frac{\ln |B^M_\sigma(R)|}{R}.
\end{equation*}
If $\mu_M=0$ we say that $M$ has {\em subexponential volume growth. }
\end{definition}

 The definition is justified by the fact that   \( \mu_M = 0 \) is equivalent to
\begin{equation*}
\limsup_{R \to \infty} \frac{|B^M_\sigma(R)|}{e^{\alpha R}} = 0, \quad \text{for all } \alpha > 0.
\end{equation*}

Notice that   having subexponential growth is a weaker assumption than being bounded by a polynomial of any degree.

The logarithmic volume growth  of a manifold $M$ is related to the bottom of the spectrum of the Laplace-Beltrami operator 
\( \Delta \) on  \( M \).

We recall that the \textit{bottom of the (essential) spectrum} \( \sigma(M) \) of \( -\Delta \) is defined as \begin{equation*}
\lambda_0(M) := \inf \{ \sigma(M) \} = \inf_{\substack {f \in C_0^\infty(M) \\  f \neq 0}} \frac{\int_M |\nabla f|^2}{\int_M f^2} 
\left(\lambda_0^{\text{ess}}(M) := \inf \{ \sigma_{\text{ess}}(M) \} = \sup_K \lambda_0(M \setminus K)\right)
\end{equation*}    

where the supremum is taken over all compact subsets \( K \subset M \).


The following theorem, due to  Brooks \cite{Br1}   gives a relation between the spectrum of the laplacian on  a manifold, and the logarithmic volume growth.
\begin{teorema}[\textbf{Brooks' Theorem}]
\label{brooks}
If \( M \) has infinite volume, then the following chain of inequalities holds:
\begin{equation*} \label{eq:brooks}
 \lambda_0(M) \leq \lambda_0(M \setminus K) \leq \lambda_0^{\mathrm{ess}}(M) \leq \frac{\mu_M^2}{4}.
\end{equation*}
\end{teorema}

\

For later use, in the case of a submanifold, we  introduce a logarithmic growth related to the second fundamental form.

  For any $p>0$, the  \textit{Total \( p \)-Curvature  function of a hypersurface  \( M \)} of a Riemannian manifold is 
   \begin{equation*}
\mathcal{T}_p(R) := \int_{B_\sigma(R)} |\phi|^p,
\end{equation*}
   
where \( \phi=A-Hg \) is the traceless second fundamental form of $M.$
  
 Then, the \textit{logarithmic growth of the Total \( p \)-Curvature} is 
   
 \begin{equation*}
\mu_{\mathcal{T}_p} := \limsup_{R \to \infty} \frac{\ln \mathcal{T}_p(R)}{R}.
\end{equation*}

\section{Subexponential volume growth}

 The main result of  this section is the positive answer  to do Carmo's question  for finite $\delta$-index $H$-hypersurfaces with subexponential volume growth in ${\mathbb R}^{n+1},$ for any $n$. 
Furthermore, we  give a  positive answer to Do Carmo's question for $\delta$-stable $H$-hypersurfaces   (for a suitable range of $\delta$ and $H$)    in ${\mathbb R}^{n+1}$ and ${\mathbb H}^{n+1}$ with $n\leq 5,$ provided the total curvature has subexponential growth.\\ In the Euclidean setting, this conclusion holds without any additional assumption on the growth of the total curvature (see \cite{tinagliazhou} and the remark at the end of Section 4 in \cite{chenhongli}). Finally, we remark that in \cite{H} the author gives a positive answer to Do Carmo's question   
in ${\mathbb H}^4$ for  finite index $H$-hypersurfaces ($H>1$) with finite topology, using the $\mu$-bubble construction. For the hyperbolic setting, the present work also provides some results in higher dimensions (see Theorem \ref{do-carmo-totalcurvature}).

We start by proving a general fact that is a direct consequence of Brooks's inequality on the bottom of the  essential spectrum (see  \cite[Theorem 9]{entropies} for the stable case).

\begin{teorema}
\label{docarmo-delta-entropy}
     There is no complete, noncompact, finite $\delta$-index $H$-hypersurface \( M \) of dimension $n$ immersed in a manifold \( N \), provided the mean curvature  \( H \) satisfies:
\[
(1 - \delta )(nH^2 + \mathrm{Ric}_N(\nu, \nu)) \geq c> \frac{\mu_M^2}{4},
\]
for a positive  constant \( c \).
\end{teorema}

\begin{proof} 
Assume such a hypersurface \( M \) exists. Since \( M \) has finite $\delta$-index, there exists a compact set \( K \subset M \) such that \( M \setminus K \) is $\delta$-stable.\\Therefore, for any \( f \in C^\infty_0(M \setminus K) \), one has:
\[
0 \leq Q_\delta(f, f) = \int_{M \setminus K} \left( |\nabla f|^2 - (1 - \delta )(|A|^2 + \mathrm{Ric}_N(\nu, \nu)) f^2 \right).
\]
\noindent But since
\[
(1- \delta)(|A|^2 + \mathrm{Ric}_N(\nu, \nu) )\geq (1- \delta) (nH^2 + \mathrm{Ric}_N(\nu, \nu)) \geq c,
\]
we obtain
\begin{equation*} \label{eq:stability}
0 \leq \int_{M \setminus K} \left( |\nabla f|^2 - c f^2 \right),
\end{equation*}
which implies
\[
\lambda_0(M \setminus K) \geq c > \frac{\mu_M^2}{4}.
\]
\noindent This contradicts Theorem \ref{brooks}.
\end{proof}

As a consequence of Theorem \ref{docarmo-delta-entropy} we get the following non existence result.

 \begin{teorema}
 \label{docarmo1}
   There are no complete non-compact finite $\delta$-index   $H$-hypersurfaces with subexponential volume growth, immersed in $N$, provided either  $N={\mathbb S}^{n+1},$  or  $N={\mathbb R}^{n+1}$ and $H\not=0$, or $N={\mathbb H}^{n+1}$ and $H>1.$
  \end{teorema} 

\begin{osservazione}
    Notice  that, in particular, Theorem \ref{docarmo1} implies a positive answer to Do Carmo's question in the Euclidean space, in the case of subexponential  volume growth, with  no restrictions on $\delta$ and on the dimension. 
\end{osservazione}

 It is worth noting that the volume growth condition on
M can be replaced by a volume growth condition on the ambient space, provided that
M is properly embedded and has bounded second fundamental form. Indeed, under these assumptions, one can construct an embedded half-tube around
M and compare the logarithmic volume growth of M with that of the ambient space (see Theorem 1 in \cite{entropies} for the construction of the half-tube and Corollary 4 therein for the comparison of logarithmic growths).
 
The next theorem is  a generalization of  \cite[Theorem 10]{entropies}.

 \begin{teorema}
 \label{properly-embedded}
     There is no complete, noncompact, $H$-hypersurface   \( M \) of dimension $n$ and with finite $\delta$-index,  that is properly embedded in a simply connected manifold \( N \), where \( M \) and \( N \) have bounded curvature, and provided the following conditions hold:
\[
H \geq c_1 > 0 \quad \text{and} \quad (1-\delta)(nH^2 + \operatorname{Ric}_N(\nu,\nu))\geq c_2 > \frac{\mu_N^2}{4},
\]
for some positive constants \( c_1, c_2 \).

 \end{teorema}
 \begin{proof}
     From Corollary 4 in \cite{entropies} we know 
     \begin{equation*}
         \mu_M \leq \mu_M^N \leq \mu_N
     \end{equation*}
and we can apply Theorem \ref{docarmo-delta-entropy}.
 \end{proof}

As a consequence one has the following positive answer to Do Carmo's question.

 \begin{teorema}
 \label{docarmo-properly-embedded}
   There are no complete non-compact finite $\delta$-index   $H$-hypersurfaces, $H\not=0,$ with bounded second fundamental form,  properly embedded in  ${\mathbb R}^{n+1}.$ 
  \end{teorema} 
\begin{proof}
One applies Theorem \ref{properly-embedded} with  $\mu_{{\mathbb R}^{n+1}}=0.$
\end{proof}

In the following we are going to relate the logarithmic  volume growth  with the  logarithmic growth of the total curvature.

First we generalize   \cite[Theorem 7]{entropies} to finite $\delta$-index $H$-hypersurfaces.

\begin{teorema}
\label{entropiatotale}
Let \( M \) be a complete, noncompact,   $H$-hypersurface of dimension $n$ with finite $\delta$-index immersed in \( N \), with $\delta < \frac{2}{n+2}$.\\ Assume \( q \in \left [0, -\delta + \sqrt{\delta^2 -\delta + \frac{2}{n}(1-\delta)} \right )\), then one has
\[
\mu _{\mathcal{T}_{2q+4}} \leq \mu_M.
\]
\end{teorema}
\begin{proof} Let $K$ be a compact set such that $M\setminus K$ is $\delta$-stable. We apply Theorem \ref{beta123} with  the following test function \( f \):
\[
f = 1 \text{ in } B_{\sigma}^M(R), \quad f = 0 \text{ in } M \setminus B_{\sigma}^M((1 + t)R),
\]
and linear in \( B_{\sigma}^M((1 + t)R) \setminus B_{\sigma}^M(R) \) for a fixed point  $\sigma\in M.$

\noindent With such choices, assuming $R$ large enough such that $K\subset B^M_\sigma(R),$  the left hand side of equation \eqref{beta123eq} is larger than
\[
\beta_1 \int_{B_\sigma^M(R)\setminus K} |\phi|^{2q+4},
\]
while the right hand side is smaller than
\[
(\beta_2 + \beta_3) |B_\sigma^M((1 + t)R)|.
\]
\noindent Then,
\begin{equation*} 
\frac{\beta_1 \int_{B_\sigma^M(R\setminus K)} |\phi|^{2q+4}}{R} \leq (\beta_2 + \beta_3) \frac{|B_\sigma^M((1 + t)R)|}{R}.
\end{equation*}

We take the logarithm and the limit of both sides, first as \( R \to \infty \), then as \( t \to 0 \), and the result follows.
\end{proof}

Next result is a generalization of \cite[Theorem 8]{entropies} to finite $\delta$-index $H$-hypersurfaces.

\begin{teorema}
\label{tracelessfinita}
   Let \( M \) be a complete, non-compact $H$-hypersurface ($H\not=0$) with  finite $\delta$-index and dimension $n$,  immersed  in a  manifold  \( N \) with constant sectional curvature \( c \). Assume moreover that  $n\leq 5$ and $\delta$, $n$ satisfy the following inequality
   \begin{equation*}
       \delta < 1- \frac{n^2}{2 (n+2)\sqrt{n-1}}.
   \end{equation*}Then, provided either:

\begin{enumerate}
    \item \( c = 0 \) or \( c = 1 \), with \( q \in [0, q_2) \), or
    \item \( c = -1 \), \( \varepsilon > 0 \), \( q \in [1, q_2 - \varepsilon] \), and \( H^2 \geq g_n(q) \),
\end{enumerate}
where $g_n$ and $\alpha_2, q_2$  are  defined in \eqref{eq:gn},\eqref{eq:alpha}and \eqref{eq:x1x2},
one has
\begin{equation*} \label{eq:entropy_condition}
\int_M |\phi|^{2q+2} < \infty \quad \text{if and only if} \quad \mu_{\mathcal T_{2q+2}} \equiv 0,
\end{equation*}
where \( \phi = A - H g \) is the traceless part of the second fundamental form. 
\end{teorema}
\begin{proof}
Notice that the condition $n\leq 5$ is needed to guarantee that $q_2>0$ as specified in the Appendix.
The implication \(  \int_M |\phi|^{2q+2} < \infty \Rightarrow \mu_{\mathcal T_{2q+2}} \equiv 0 \) follows directly from the definition of logarithmic growth.
\\Conversely, since inequality \eqref{gammad} holds on \( M \), we can apply part (1) of  \cite[Theorem 3]{entropies} with \( w = |\phi|^{2q+2} \) to conclude the result.

\end{proof}

The next result is a positive answer to  Do Carmo's question
for $H$-hypersurfaces with finite $\delta$-index in ${\mathbb R}^{n+1}$ and ${\mathbb H}^{n+1}$, provided a growth assumption on the total curvature is satisfied (see  \cite[Theorem 11]{entropies} for the stable case). 

 \begin{teorema}\label{do-carmo-totalcurvature}
There is no complete, noncompact, $\delta$-stable, $H$-hypersurface \( M\) of dimension $n\le 5$, immersed in a manifold \( N \) with \( \mu_{\mathcal T_{2q+2}} = 0 \), provided 
\begin{equation*}
    \delta< 1- \frac{n^2}{2 (n+2)\sqrt{n-1}}
\end{equation*}
and either:
\begin{enumerate}
    \item \( N = \mathbb{R}^{n+1} \), \( q \in [0, q_2) \), \( H > 0 \), or
    \item \( N = \mathbb{H}^{n+1} \), \( \varepsilon > 0 \), \( q \in [0, \operatorname{min}\{\alpha_2, q_2\} - \varepsilon] \), \( H^2 > g_n(q) \)
    where $g_n$ and $\alpha_2, q_2$ are defined in\eqref{eq:gn},\eqref{eq:alpha}and \eqref{eq:x1x2}.
\end{enumerate}
\end{teorema}

\begin{proof}

First, notice that by Theorem \ref{tracelessfinita}, the hypothesis \( \mu_{\mathcal T_{2q+2}} = 0 \) implies \( \int_M |\phi|^{2q+2} < \infty \).
\\One then applies Theorem \ref{gammadth} with  test function $f=1$ on $B_{\sigma}^M(R)$, $f=0$ on $B_{\sigma}^M(2R)\setminus B_{\sigma}^M(R)$ and $|\nabla f| \leq \frac{1}{R}$ to deduce that \( M \) is totally umbilic.
\\ \textbf{Case (1):} Since \( N = \mathbb{R}^{n+1} \), it follows that \( M \) is contained either in a sphere or in a hyperplane. As \( M \) is complete and noncompact, it must be contained in a hyperplane, which implies \( H = 0 \), contradicting the hypothesis \( H > 0 \).
\\ \textbf{Case (2):} In the case where \( N = \mathbb{H}^{n+1} \), the totally umbilic hypersurfaces are either geodesic spheres, horospheres, or equidistant hypersurfaces. The condition \( H^2 > g_n(q) \geq 1 \) implies that \( M \) can only be contained in a geodesic sphere. However, as \( M \) is complete and noncompact, this leads again to a contradiction.\\ Hence, in both cases, such an \( M \) cannot exist.

\end{proof}

\section{Ricci curvature bounded below}
 In this section, we prove that finite $\delta$-index $H$-hypersurfaces in $\mathbb{R}^{n+1}$ satisfying a suitable  Ricci curvature bound are minimal. 
 More precisely we will require that 

\begin{equation*}
\mathrm{Ric}_M \geq
- \frac{C(1-\delta)}{n-1} |A|^2 g.
\end{equation*}

where $C$ is a positive constant smaller than 4. 
As already mentioned in the Introduction,
the previous  condition  is quite natural since,   for any  hypersurface in a manifold with non negative sectional curvature, the  weaker condition

\begin{equation}\label{ineq-ricci}
\mathrm{Ric}_M \ge
- \frac{\sqrt{n-1}}{2} |A|^2 g
\end{equation}

always holds (see, for example,  \cite[Proposition 2.2]{chenhongli}).

See also \cite[Proposition 2]{ShXu} for the proof of an inequality, stronger than \eqref{ineq-ricci}, satisfied by   $\mathrm{Ric_M}$  when the ambient space has Ricci curvature bounded below by a constant.\\

We state again the main theorem of this section.
\begin{teorema}
\label{th-docarmo}
   Assume that $M$ is a complete, non-compact, finite $\delta$-index $H$-hypersurface immersed ${\mathbb R}^{n+1}.$ Assume that 
   \begin{equation*}
    \operatorname{Ric}_M \ge- \frac{C(1-\delta)}{n-1}|A|^2g    \end{equation*}
    where $C\in(0,4).$ 
    Then, $M$ is minimal, i.e. $H=0.$  
\end{teorema}

In the proof, we employ a conformal method based on estimates along a suitably chosen geodesic, that is, in a fixed direction of $M$. This approach is rather classical for results of do Carmo's type (see, for instance, \cite{cheng}, \cite{elbertnelliros}, \cite{tinagliazhou}).

In our setting, the hypothesis on the Ricci curvature completely removes the need for any dimensional assumption. In particular, our result holds for arbitrary dimension $n$. Moreover, no restriction on $\delta$ is required.

The geodesic used in the argument is constructed as in \cite[Proposition 2.1]{cheng}. For the reader’s convenience, and since it plays a key role in what follows, we state it here in the form of a Lemma.

\begin{lemma}{\em \cite[Proposition 2.1, claim 1]{cheng}}
\label{th:cheng}
Let $(M, g)$ be a complete Riemannian manifold and let $K \subset M$ be a compact subset. Let $u >0$ be a positive function.
Define the conformal change of metric $\tilde g= u^ {2k}g$. Then, there exists a minimizing geodesic $\gamma (s): [0, +\infty) \to M \setminus K$ in the metric $\tilde g$,
where the parameter $s$ is arc length in the metric $g$. Moreover, $\gamma$ has infinite length in the metric $g$.
\end{lemma}

\textit{Proof of Theorem 3.1.} The proof is by contradiction. Assuming that  $M$ is not minimal we are going to prove that $M$ is compact.
   
By the finite index assumption, there exists a compact set $K$ in $M$  and a positive function $u$ such that 
$J_\delta u=0,$ on $M \setminus K$.

Let $0<k<\frac{4}{n-1}$ and define a conformal change of the metric as
\[
\tilde g = u^{2k} g.
\]

Using Lemma \ref{th:cheng}, we know that there exists a $\tilde g$-minimizing geodesic $\gamma: [0, +\infty) \to M \setminus K$ that has infinite length in the metric $g$. 

Let $R$ and $\tilde R$ be the
curvature tensors of $M$ in the metrics $g$ and $\tilde g$ respectively. Choose an orthonormal basis
\[
\left\{ e_1 = \frac{\partial\gamma}{\partial s},\ e_2, \dots, e_n\right\}
\]
for $g$ such that $e_2, \dots, e_n$ are parallel along $\gamma$ and let $e_{n+1} = \nu$. This
yields an orthonormal basis
\[
\left\{ \tilde e_1 = \frac{\partial\gamma}{\partial\tilde s}  = u^ke_1,\; \tilde e_2 = u^k e_2,\dots,\tilde e_n = u^k e_n \right\}
\]
for $\tilde g$. We denote by
$ R_{11}$ (respectively $\tilde R_{11}$) the Ricci curvature in the direction of $e_1$ for the metric
$g$ (respectively $\tilde g$).

Since $\gamma$ is $ \tilde g$-minimizing, using the second variation formula for the energy, we have for any $l>0$ and $\varphi$ a smooth function on 
$\gamma$  such that $\varphi(0)=\varphi(l)=0$

\begin{equation}
\label{eq:secondvariation}
\int_0^ l \left( (n-1)\left(\frac{d\varphi}{d \tilde s}\right)^2 - \tilde R_{11}\varphi^2 \right) d \tilde s \ge 0,
\end{equation}
for any smooth function $\varphi$ with compact support on $[0, + \infty)$. We apply the formula for Ricci curvature
under conformal change of metric (see for instance the appendix in \cite{elbertnelliros} for a full computation) and use $J_\delta u=\Delta u+(1-\delta)|A|^2 u$ to get 

\[
\tilde R_{11}
= u^{-2k} \left\{R_{11}
- k(n-2)(\log u)_{ss}
- k\frac{ J_\delta u-(1-\delta)|A|^2u}{u}
+ k\frac{|\nabla u|^2}{u^2} \right \}.
\]
Combining the last  equality with \eqref{eq:secondvariation} gives (by abuse of notation, we
denote again $\varphi$ the composition $\varphi \circ \tilde s$) 
\begin{equation*}
 \begin{split}   
(n-1) & \int_0^l(\varphi_s)^2 u^{-k} ds \\
& \ge
\int_0^l\varphi^2 u^{-k}
\left(R_{11}
- k(n-2)(\log u)_{ss}
- k\frac{  J_\delta u}{u}+k(1-\delta)|A|^2
+ k\frac{|\nabla u|^2}{u^2} \right) \,ds.
\end{split} \end{equation*}

Now, replace $\varphi$ by $\varphi u^{k/2}$ to eliminate the $u^{k}$ in the denominator. Rearranging terms we get
\begin{equation*}
    \begin{split}
    -k  \int_0^l \varphi^2 \frac{ J _\delta u}{u} \, ds&+ \int_0^l\varphi^2 
\left(R_{11}+ k (1-\delta)|A|^2
\right) \, ds +k\int_0^l \varphi^2\frac{|\nabla u|^2}{u^2}  \\
 & \le (n-1)\int_0^l (\varphi_s)^2 \, ds
+ k(n-1)\int_0^l \varphi\varphi_s u_s u^{-1} \, ds
 \\
&+ \frac{k^2(n-1)}{4}\int_0^l \varphi^2 u_s^2 u^{-2} \, ds- \int_0^l \varphi^2 
k(n-2)(\log u)_{ss}
 \, ds.
  \end{split}
\end{equation*}

Integrating by parts
\[
\int_0^l \varphi^2 (\log u)_{ss}\, ds
    = -2 \int_0^l \varphi \varphi_s u_s u^{-1} \, ds.
\]
Finally, using that $J-\delta u=0$

\begin{equation*}
    \begin{split}
   \int_0^l \varphi^2 
\left(R_{11}+ k (1-\delta)|A|^2
\right) \, ds & \le   (n-1)\int_0^l (\varphi_s)^2 \, ds  - k(n-3 )\int_0^l \varphi\varphi_s u_s u^{-1} \, ds  \\ &+ k \left[\frac{(n-1)k}{4}-1\right]\int_0^l\varphi^2 (u_s)^2u^{-2} \, ds
    \end{split}
\end{equation*}

Now,  using that $\operatorname{Ric}_M \geq -\frac{C (1- \delta)}{n-1}|A|^2g$ and  applying  the weighted Young’s inequality to the second term on the right-hand side $\bigg(\left(\frac{a^2}{\varepsilon}+ \varepsilon b^2 \right)\ge -2ab$, with $a= \frac{1}{2} k (n-3) \varphi_s$ and $b= \frac{u_s}{u} \varphi \bigg)$.\\ We obtain
\begin{equation*}
    \begin{split}
 \int_0^l \varphi^2 
\left((1-\delta)\left[k -\frac{C}{n-1}\right]|A|^2
\right) \, ds  &\leq \left (k \left[\frac{(n-1)k}{4}-1\right]+\varepsilon\right)\int_{0}^{+\infty}\varphi^2(u_s)^2{u^{-2}} \\&+ \left((n-1) + \frac{k^2 (n-3)^2}{4 \varepsilon}\right) \int_{0}^{+\infty} (\varphi_s)^2
\end{split}
\end{equation*}
Now,  as $C<4$ it is possible to choose $k$ such that $\frac{C}{n-1}<k< \frac{4}{n-1}$. Then the coefficient on the left hand side is positive. Moreover,
since  $k< \frac{4}{n-1}$, choosing $\varepsilon$ small enough, the first coefficient on the right hand side is negative. 
\\Thus, there exists a positive  constant $B$ such that 
\[
nH^2\int_{0}^l \varphi^2 \leq \int_{0}^l \varphi^2 |A|^2 \leq B \int_{0}^l (\varphi_s)^2.
\]

where we have used that $|A|^2\geq nH^2.$ 

Now,  take $ \varphi:=\sin\left(\frac{\pi s}{l}\right)$. Since \begin{equation*}
  \int_{0}^l \sin^2\left(\frac{\pi s}{l}\right)= \int_{0}^l \cos^2\left(\frac{\pi s}{l}\right) =\frac{l}{2}, 
\end{equation*} 
one gets
\[
\frac{nH^2 l}{2}\leq \frac{B \pi^2}{2l}
\]

As $H \ne 0$, the previous inequality gives a bound on $l,$ which is a contradiction with the fact that $\gamma$ has infinite length. \\
\qed

Let ${N}$ be a Riemmannian manifold  of dimension $n+1$ and denote by $\operatorname{sec}(N)$ the infimum of the sectional curvatures of $N$. We have the following result.

\begin{teorema}
\label{docarmo-riccibounded-general}   
Assume that $M$ is a complete, noncompact, finite $\delta$-index $H$-hypersurface immersed in a Riemannian manifold  $N$.  Assume    \begin{equation*}
    \operatorname{Ric}_M \ge- \frac{C(1-\delta)}{n-1}|A|^2g,   \ \ \operatorname{sec}(N)>-\infty 
    \end{equation*}
    where $C\in(0,4).$    Then the mean curvature of  $M$  satisfies 
    \begin{equation*}    
H^2\le-\frac{4\,\operatorname{sec}(N)}{4-C}.
    \end{equation*}
\end{teorema}

\begin{proof} The proof follows the same steps as the proof of Theorem \ref{th-docarmo}, keeping into account that $J_\delta =\Delta+(1-\delta)(|A|^2+Ric_N(\nu, \nu))$ where $\nu$ is the unit normal vector field to $M.$ 
\end{proof}

 \appendix 
\renewcommand{\thesection}{\Alph{section}}

\makeatletter
\renewcommand{\section}{%
  \@startsection{section}{1}{\z@}%
  {-3.5ex \@plus -1ex \@minus -.2ex}%
  {2.3ex \@plus.2ex}%
  {\normalfont\Large\bfseries}%
}
\renewcommand{\@seccntformat}[1]{Appendix~\thesection:~}
\makeatother
 \section{Caccioppoli's type inequalities}

This section is devoted to establishing some Caccioppoli-type inequalities for constant mean curvature $\delta$-stable hypersurfaces. For the reader’s convenience, we present the most significant proofs, which closely follow the arguments in \cite{caccioppoli}.\\

Throughout the Appendix, we assume that the ambient manifold $N$ satisfies the following assumptions: we denote by \( \operatorname{sec}(X,Y) \) the sectional curvature of \( N \) for the 2-plane generated by \( X, Y \in TN \). We assume that there exist constants \( K_1, K_2 \in \mathbb{R} \) such that
\[
K_2 \leq \operatorname{sec}(X,Y) \leq K_1, \quad \text{for all } X, Y \in TN.
\]

where \( R \) is the Riemann curvature tensor  of $N$.
Furthermore, we assume that the derivative of the curvature tensor is bounded. That is, there exists a constant \( K' \) such that, for any elements \( e_i, e_j, e_k, e_s, e_t \) of a local orthonormal frame, one has
\[
\left| \langle (\nabla_{e_t} R)(e_i, e_j)e_k, e_s \rangle \right|^2 \leq K'.
\]

Let \( M \) be a noncompact  hypersurface immersed in \( N \) with  constant mean curvature $H$ oriented by a unit normal vector field  \( \nu \) such that, when $H\not=0, $  then   \( \vec{H} = H \nu \), with \( H > 0 \).  Recall that  \( \phi \)  is the traceless part of the second fundamental form, i.e. $\phi := A - H g$,
where \( g \) is the induced metric on \( M \).
\\

The following result generalizes \cite[Theorem 4.1]{caccioppoli} to the case of finite $\delta$-index $H$-hypersurfaces,
 for any $\delta\in(0,1).$

\begin{teorema}
\label{abcdefg}
    Let \( M \) be a complete, noncompact $H$-hypersurface of dimension $n$ with finite $\delta$-index, immersed in a manifold \( N \). Then, there exists a compact subset \( K \subset M \) such that, for any \( q > -\frac{n+2}{2n} \) and for any \( f \in C_0^\infty((M \setminus K)^+) \), one has:

\begin{equation}\begin{split}
\label{teoremaabcdefg}
&\int_{(M \setminus K)^+} f^2 |\phi|^{2q+2} (A_\delta |\phi|^2 + B H |\phi| + C_\delta H^2 + E_\delta) \leq \\
&\leq D \int_{(M \setminus K)^+} |\phi|^{2q+2} |\nabla f|^2 
+ F \int_{(M \setminus K)^+} f^2 |\phi|^{2q+1} 
+ G \int_{(M \setminus K)^+} f^2 |\phi|^{2q}.
\end{split}\end{equation}
Where the constants are given by:
\begin{align*}
A_\delta &=
\left( \frac{2}{n}(1+\varepsilon) + (2q+1) - \tilde\varepsilon \right)(1-\delta)
- (q+1)(q+1+\tilde\varepsilon), \\[0.3em]
B &=
- a_1 (q+1)(q+1+\tilde\varepsilon), \\[0.3em]
C_\delta &=
n\left( \frac{2}{n}(1+\varepsilon) + (2q+1) - \tilde\varepsilon \right)(1-\delta)
+ n(q+1)(q+1+\tilde\varepsilon), \\[0.3em]
D &=
\frac{q+1+\tilde\varepsilon}{\tilde\varepsilon}
\left( \frac{2}{n}(1+\varepsilon) + (3q+2) - \tilde\varepsilon \right), \\[0.3em]
E_\delta &=
a_2 (q+1)(q+1+\tilde\varepsilon)
+ nK^2 \left( \frac{2}{n}(1+\varepsilon) + (2q+1) - \tilde\varepsilon \right)(1-\delta), \\[0.3em]
F &=
2nK'(q+1)(q+1+\tilde\varepsilon), \\[0.3em]
G &=
- a_3 (q+1)(q+1+\tilde\varepsilon).
\end{align*}

for any \( \varepsilon, \tilde{\varepsilon} > 0 \) and

\[
\begin{aligned}
a_1&= \frac{n(n-2)}{\sqrt{n(n-1)}},\  a_2= \frac{n(K_2 - K_1)H}{2}+ n(2K_2 - K_1),\\ 
a_3&= \frac{n^2 H(K_2 - K_1)}{2} -\frac{n(n-1)}{2\varepsilon}(K_1 - K_2)^2.
\end{aligned}
\]

Moreover, if in addition \( q \geq 0 \), then we can replace \( (M \setminus K)^+ \) with \( M \setminus K \), and if \( M \) is $\delta$-stable, then \( K = \emptyset \).

\end{teorema}

\begin{proof}
From Theorem \cite[Theorem 3.1]{caccioppoli},
\begin{equation}
\label{simona1a2a3}
|\phi| \, \Delta |\phi| \geq \frac{2}{n(1+\varepsilon)} |\nabla |\phi||^2 - |\phi|^4 - a_1 H |\phi|^3 + |\phi|^2(nH^2 + a_2) - 2nK' |\phi| + a_3.
\end{equation}
By the fact that $M$ has finite $\delta$- index, we know there exists a compact subset \( K \subset M \) such that \( M \setminus K \) is $\delta$- stable. Notice that if \( M \) is $\delta$- stable, then \( K = \emptyset \).
\\Multiplying inequality \eqref{simona1a2a3} by \( |\phi|^{2q} f^2 \), with \( f \in C_0^\infty((M \setminus K)^+) \), and integrating over \( (M \setminus K)^+ \), we obtain:
\begin{align}
\label{comb1}
& -(2q+1) \int_{(M \setminus K)^+} |\phi|^{2q} f^2 |\nabla |\phi||^2 
- 2 \int_{(M \setminus K)^+} |\phi|^{2q+1} f \langle \nabla f, \nabla |\phi| \rangle 
+ \int_{(M \setminus K)^+} |\phi|^{2q+4} f^2 
\nonumber \\
& + a_1 H \int_{(M \setminus K)^+} |\phi|^{2q+3} f^2 
- (nH^2 + a_2) \int_{(M \setminus K)^+} |\phi|^{2q+2} f^2 
+ 2nK' \int_{(M \setminus K)^+} |\phi|^{2q+1} f^2 
\nonumber \\
& - a_3 \int_{(M \setminus K)^+} |\phi|^{2q} f^2 
\geq \frac{2}{n(1+\varepsilon)} \int_{(M \setminus K)^+} |\nabla |\phi||^2 |\phi|^{2q} f^2.
\end{align}
We observe that, since we allow \( q \) to be negative, we restrict ourselves to the subset \( (M \setminus K)^+ \). If \( q \geq 0 \), then \( f \) can be taken in \( C_0^\infty(M \setminus K) \) as well as the domain of integration of all the integrals in the rest of the proof.
\\Young's inequality applied to the second term on the left hand side gives:
\begin{equation*}
|2 |\phi|^{2q+1} f \langle \nabla f, \nabla |\phi| \rangle| \leq 2 (|\phi|^q f |\nabla |\phi||)(|\phi|^{q+1} |\nabla f|) 
\leq \tilde{\varepsilon} |\phi|^{2q} f^2 |\nabla |\phi||^2 + \frac{1}{\tilde{\varepsilon}} |\phi|^{2q+2} |\nabla f|^2.
\end{equation*}
for some \( \tilde{\varepsilon} > 0 \).
\\Then,
\begin{align*}
& \left( \frac{2}{n(1+\varepsilon)} + (2q+1) - \tilde{\varepsilon} \right) \int_{(M \setminus K)^+} |\phi|^{2q} f^2 |\nabla |\phi||^2 
\leq \frac{1}{\tilde{\varepsilon}} \int_{(M \setminus K)^+} |\nabla f|^2 |\phi|^{2q+2} 
\nonumber \\
& + \int_{(M \setminus K)^+} \left( |\phi|^2 + a_1 H |\phi| - (n H^2 + a_2) \right) |\phi|^{2q+2} f^2 
+ 2nK' \int_{(M \setminus K)^+} |\phi|^{2q+1} f^2 \\
& - a_3 \int_{(M \setminus K)^+} |\phi|^{2q} f^2.
\end{align*}
The $\delta$- stability inequality restricted to \( (M \setminus K)^+ \) gives, for each \( \psi \in C_0^\infty((M \setminus K)^+) \):
\begin{equation*}
\int_{(M \setminus K)^+} |\nabla \psi|^2 
\geq \int_{(M \setminus K)^+} (1- \delta)(|A|^2 + \text{Ric}_N(\nu, \nu)) \psi^2 
\geq \int_{(M \setminus K)^+} (1-\delta) (|\phi|^2 + nH^2 + nK_2) \psi^2.
\end{equation*}
and taking \( \psi = f |\phi|^{q+1} \),
\begin{align*}
& \int_{(M \setminus K)^+} |\phi|^{2q+2} |\nabla f|^2 
+ (q+1)^2 \int_{(M \setminus K)^+} |\phi|^{2q} f^2 |\nabla |\phi||^2 
+ 2(q+1) \int_{(M \setminus K)^+} |\phi|^{2q+1} f \langle \nabla f, \nabla |\phi| \rangle 
\nonumber \\
& \geq (1- \delta) \int_{(M \setminus K)^+} |\phi|^{2q+4} f^2 + n(H^2 + K_2)(1- \delta) \int_{(M \setminus K)^+} |\phi|^{2q+2} f^2.
\end{align*}
And again by Young's inequality:
\begin{equation*}
2 \left| \int_{(M \setminus K)^+} |\phi|^{2q+1} f \langle \nabla f, \nabla |\phi| \rangle \right|
\leq \tilde{\varepsilon} \int_{(M \setminus K)^+} |\phi|^{2q} f^2 |\nabla |\phi||^2 
+ \frac{1}{\tilde{\varepsilon}} \int_{(M \setminus K)^+} |\phi|^{2q+2} |\nabla f|^2.
\end{equation*}
we obtain
\begin{align}
\label{comb2}
& -(q+1)(q+1 + \tilde{\varepsilon}) \int_{(M \setminus K)^+} |\phi|^{2q} f^2 |\nabla |\phi||^2 
\leq \left( 1 + \frac{(q+1)}{\tilde{\varepsilon}} \right) \int_{(M \setminus K)^+} |\nabla f|^2 |\phi|^{2q+2}
\nonumber \\
& \quad - (1- \delta)\int_{(M \setminus K)^+} f^2 |\phi|^{2q+2} \left( |\phi|^2 + n(H^2 + K_2) \right).
\end{align}
Now, we want to get rid of the term \( |\phi|^{2q} f^2 |\nabla |\phi||^2 \). To achieve this, the idea is to do a suitable linear combination of the inequalities \eqref{comb1} and \eqref{comb2}. One needs 
\( (q+1)(q+1 + \tilde{\varepsilon}) > 0 \) and
\( \frac{2}{n(1+\varepsilon)} + (2q+1) - \tilde{\varepsilon} > 0 \) 
for some \( \tilde{\varepsilon} \) sufficiently small. These two conditions are satisfied if
\begin{equation*}
q >- \frac{n+2}{2n}.
\end{equation*}
Summing
 \( (q+1)(q+1 + \tilde{\varepsilon}) \cdot \eqref{comb1} \) and
\( \left( \frac{2}{n(1+\varepsilon)} + (2q+1) - \tilde{\varepsilon} \right) \cdot \eqref{comb2} \), we get:
\begin{align*}
0 \leq & \left[ \frac{(q+1)(q+1 + \tilde{\varepsilon})}{\tilde{\varepsilon}} + \left( \frac{2}{n(1+\varepsilon)} \right) \left( 1 + \frac{(q+1)}{\tilde{\varepsilon}} \right) \right] \int_{(M \setminus K)^+} |\nabla f|^2 |\phi|^{2q+2}
\nonumber \\
& + \int_{(M \setminus K)^+} \bigg[ (q+1)(q+1 + \tilde{\varepsilon})(|\phi|^2 + a_1 H |\phi| - (n H^2 + a_2)) 
\nonumber \\
&  \quad - \left( \frac{2}{n(1+\varepsilon)} + (2q+1) - \tilde{\varepsilon} \right)(|\phi|^2 + n H^2 + n K_2)(1- \delta) \bigg] |\phi|^{2q+2} f^2
\nonumber \\
& + 2nK' (q+1)(q+1 + \tilde{\varepsilon}) \int_{(M \setminus K)^+} |\phi|^{2q+1} f^2 
- a_3 (q+1)(q+1 + \tilde{\varepsilon}) \int_{(M \setminus K)^+} |\phi|^{2q} f^2.
\end{align*}
Taking $A_\delta, B, C_\delta, D, E_\delta, F$ and $G$ as in the statement of the theorem we conclude the proof.
\end{proof}

The following result generalizes \cite[Theorem 4.2]{caccioppoli} to finite  $\delta$-index $H$-hypersurfaces, provided 
$\delta < \frac{2}{n+2}$.

\begin{teorema}
\label{teo-sdelta}
Let \( M \) be a complete non-compact $H$-hypersurface immersed  in a manifold with constant curvature \( c \). Assume \( M \) has finite $\delta$-index, with $\delta < \frac{2}{n+2}$. Then there exists a geodesic ball \( B^M(R_0) \subset M \) such that, for any \( q \) satisfying  \( 0 \leq  q <  - \delta + \sqrt{\delta^2 - \delta + \frac{2}{n}(1- \delta)} \),
\[
\int_{M \setminus B^M{(R_0)}} |\phi|^{2q+4} \leq S_ \delta \int_{M \setminus B^M{(R_0)}} |\phi|^{2q+2}
\]
for some positive constant \( S_\delta \). Moreover, if \( M \) is $\delta$-stable, then we can choose \( B^M{(R_0)} = \emptyset \).
\end{teorema}

The following result generalizes \cite[Theorem 5.4]{caccioppoli} to the case of finite $\delta$-index $H$-hypersurfaces, provided $\delta$ satisfies a suitable inequality. 

\begin{teorema}
\label{gammadth}
    Let \( M \) be a complete non-compact $H$-hypersurface of dimension $n$ immersed in a manifold with constant curvature \( c \).  
Assume \( M \) has finite $\delta$-index and $\delta$ satisfies the inequality
\begin{equation*}
  \delta <  1- \frac{n^2}{2(n+2) \sqrt{n-1}}. 
\end{equation*}
Then there exist a compact subset \( K \subset M \) and a constant \( \gamma > 0 \) such that, for any \( f \in C_0^\infty(M \setminus K) \),
\begin{equation}
\label{gammad}
\gamma \int_{M \setminus K} f^2 |\phi|^{2q+2} \leq D \int_{M \setminus K} |\phi|^{2q+2} |\nabla f|^2. 
\end{equation}
provided either:
\begin{itemize}
    \item[(1)] \( c = 0 \) or \( c = 1 \), and \( q \in [0, q_2) \), or
    \item[(2)] \( c = -1 \), \( \varepsilon > 0 \), \( q \in [1, \operatorname{min}\{\alpha_2, q_2\} - \varepsilon] \), and \( H^2 \geq g_n(q) \),
\end{itemize}

where $g_n,$ $q_1$ and $q_2$ are defined in the proof.
\end{teorema}
\begin{proof}
Note that in this case
\[
K_1 = K_2 = c, \quad K' = 0, \quad a_1 = \frac{n(n - 2)}{\sqrt{n(n - 1)}}, \quad a_2 = nc, \quad a_3 = 0.
\]
Furthermore, we can take \( \varepsilon = 0 \) in Kato’s inequality. Therefore, we get
\[
F = 0, \quad G = 0,
\]
\[
A_\delta = \left (\frac{2}{n} + (2q + 1) - \tilde{\varepsilon} \right )(1- \delta) - (q + 1)(q + 1 + \tilde{\varepsilon}),
\]
\[
B = -a_1 (q + 1)(q + 1 + \tilde{\varepsilon}),
\]
\[
C_\delta =n \left( \frac{2}{n(1+\varepsilon)}+(2q+1)-\tilde{\varepsilon}\right )(1 - \delta) +n(q+1)(q+1 + \tilde{\varepsilon})
\]
\[
E_\delta = nc(q+1)(q+1+ \tilde{\varepsilon})+nc  \left( \frac{2}{n}  +2q + 1- \tilde{\varepsilon} \right)(1-\delta)= cC_\delta.
\]

Then inequality \eqref{teoremaabcdefg} becomes (note that we are assuming \( q \geq 0 \)):
\begin{equation}
\int_{M \setminus K} f^2 |\phi|^{2q+2} \left( A_\delta |\phi|^2 + BH|\phi| + C_\delta (H^2 + c) \right) \leq D \int_{M \setminus K} |\phi|^{2q+2} |\nabla f|^2. 
\end{equation}
Define:
\[
\beta_7 = A_\delta|\phi|^2 + BH|\phi| + C_\delta(H^2 + c)
\]
and we must find conditions on \( A_\delta, B, C_\delta, c, H \) such that $\beta_7$ is positive. As before, it is enough to do the computations for \( \tilde{\varepsilon} = 0 \) because all quantities are continuous in \( \tilde{\varepsilon} \).\\We now consider cases $c=0,1,-1$ and $H=0, H \neq 0$ separately.
\\\textbf{(1) \( c = 0 \)}  
\begin{itemize}
    \item \( H = 0 \). In this case \( \beta_7 = A_\delta |\phi|^2 \). We only need \( A_\delta > 0 \), i.e. $q \in (\alpha_1, \alpha_2)$, where $\alpha_1, \alpha_2$ are defined as
    \begin{equation}
    \label{eq:alpha}
        \alpha_1= - \delta - \sqrt{\delta^2- \delta + \frac{2}{n}(1- \delta)}, \quad \alpha_2=  -\delta + \sqrt{\delta^2- \delta + \frac{2}{n}(1- \delta)}
    \end{equation}
    \item \( H \neq 0 \). Then
    \[
    \beta_7 = A_\delta |\phi|^2 + BH|\phi| + C_\delta H^2, 
    \]
    and \( \beta_7 > 0 \) for any \( |\phi| \) if and only if \( A_\delta > 0 \) and the discriminant is negative:
    \[
    H^2 \left( B^2 - 4A_\delta C_\delta \right) < 0.
    \]
   and computing the discriminant:
    \[
    \Delta_0 = n H^2 \left( \frac{ n^2}{n - 1} (q+1)^4- 4 \left(2(q+1)-\frac{n - 2}{n} \right)^2(1-\delta)^2\right),
    \]
    the inequality becomes:
    \begin{equation*}
    (q+1)^4 \frac{n^2}{n-1} <  4\left( 2(q+1) - \frac{n -2}{n} \right)^2 (1- \delta)^2.
    \end{equation*}
    From the positivity of $A_\delta$ the latter is equivalent to
    \begin{equation*}
      (q+1)^2 \frac{n}{\sqrt{n-1}} <  2\left( 2(q+1) - \frac{n - 2}{n} \right) (1- \delta).   
    \end{equation*}
    The discriminant of this polynomial is positive if and only if 
    \begin{equation*}
        \delta < -\frac{(n-2)}{2\sqrt{n-1}}+1.
    \end{equation*}
   Moreover, the inequality holds if and only if $q \in (q_1, q_2)$, where $q_1$ and $q_2$ are defined as 
    \begin{equation}
    \label{eq:x1x2}
\begin{aligned}
q_1 &= \frac{2\sqrt{n-1}}{n}(1-\delta)\left(1-\sqrt{1-\frac{n-2}{2\sqrt{n-1}(1-\delta)}}\right)-1,\\[4pt]
q_2 &= \frac{2\sqrt{n-1}}{n}(1-\delta)\left(1+\sqrt{1-\frac{n-2}{2\sqrt{n-1}(1-\delta)}}\right)-1.
\end{aligned}
\end{equation}

\end{itemize}

\textbf{(2) \( c = -1 \)}
\begin{itemize}
    \item \( H \geq 0 \): Then
    \[
    \beta_7 = A_\delta |\phi|^2 + BH|\phi| + C_\delta (H^2 - 1).
    \]
    The discriminant is:
    \[
    \Delta_{-1} = (B^2 - 4A_\delta C_\delta)H^2 + 4A_\delta C_\delta.
    \]
    The only case with no condition on \( |\phi| \) is \( A_\delta > 0 \) and \( \Delta_{-1} < 0 \). \\One needs \( q \in (\operatorname{max}\{\alpha_1, q_1\},\operatorname{min}\{\alpha_2, q_2\}) \) and \( H^2 > g_n(q) \), where \( g_n \) is defined as
    \begin{equation}
    \label{eq:gn}
g_n(q) = \frac{(2(q+1) - \gamma_n)^2(1- \delta)^2- (q+1)^4}{(2(q+1) - \gamma_n)^2(1- \delta)^2 - \mu (q+1)^4}
    \end{equation}
   for $\gamma_n = \frac{n - 2}{n}$, $ \mu = \frac{n^2}{4(n - 1)}$.\\
    Since \( \sup g_n = +\infty \), we restrict to \( q \in [\operatorname{max}\{\alpha_1, q_1\} + \varepsilon, \operatorname{min}\{\alpha_2, q_2\} - \varepsilon] \).
\end{itemize}

\textbf{(3) \( c = 1 \)}
\begin{itemize}
    \item \( H = 0 \): Then
    \[
    \beta_7 = A_ \delta |\phi|^2 + C_\delta,
    \]
    which is positive for all \( |\phi| \) if \( q \in [\alpha_1, \alpha_2] \).
    
    \item \( H \neq 0 \): Then
    \[
    \beta_7 = A_\delta |\phi|^2 - BH|\phi| + C_\delta (H^2 + 1),
    \]
    which is positive for all \( |\phi| \) if \( A_\delta > 0 \) and
    \[
    \Delta_1 = H^2(B^2 - 4A_\delta C_\delta) - 4A_\delta C_\delta < 0.
    \]
    These hold for all \( H \) if \( q \in (q_1, q_2) \). While, for $q \notin (q_1, q_2)$ one needs
    \[
    H^2 \leq \frac{4A_\delta C_\delta}{B^2 - 4A_\delta C_\delta} = -g_n(q).
    \]
\end{itemize}
\end{proof}
\begin{osservazione}
    Note that the range of $q$ implies the restriction on $\delta$ in the statement. Indeed, we need to have  $- \delta + \sqrt{\delta^2 - \delta + \frac{2}{n}(1- \delta)} > 0 $, that implies $\delta < \frac{2}{n+2}$.\\Moreover $q_2$ should be greater or equal than zero. This imposes some addictional restrictions on $\delta$. Indeed, from a direct computation:
    \begin{equation*}
        \delta < 1- \frac{n^2}{2 \sqrt{n-1}(n+2)}.
    \end{equation*}
    Note that, in order to have some $\delta$ satisfying this inequality one should ask for $n \le 5$.
    Observe also that \begin{equation*}
        1-\frac{n^2}{2 \sqrt{n-1}(n+2)} \le \frac{2}{n+2}
    \end{equation*}
    for all $n \geq 2$. 
\end{osservazione}

The following result is a generalization of \cite[Theorem 5.1]{caccioppoli} to $\delta$-stable hypersurfaces.

\begin{teorema}
\label{beta123}
Let $M$ be a complete non-compact hypersurface immersed with constant mean curvature $H$ in a manifold $N$. Assume $M$ has finite $\delta$-index. Then, there exist a compact subset $K$ of $M$ (which is empty if $M$ is stable) and constants $\beta_1, \beta_2, \beta_3$, such that for every $f \in C_0^\infty(M \setminus K)^+$,
\begin{equation}
\label{beta123eq}
    \beta_1 \int_{(M \setminus K)^+} f^{2q+4} |\phi|^{2q+2} \leq \beta_2 \int_{(M \setminus K)^+} |\nabla f|^{2q+4} + \beta_3 \int_{(M \setminus K)+} f^{2q+4}.
\end{equation}
Moreover, the constant $\beta_1$ is positive if and only if 
\begin{align*}
    - \delta - \sqrt{\delta^2 - \delta + \frac{2}{n}(1-\delta)} < q <   -\delta + \sqrt{\delta^2 - \delta + \frac{2}{n}(1-\delta)} 
\end{align*}
with $\delta \le \frac{2}{n}$.
Moreover, if $q \geq 0$, we can replace $(M \setminus K)^+$ with $M \setminus K$.
\end{teorema}

{Barbara Nelli}

 {ORCID: 0000-0002-0879-8390}

{Dipartimento di Ingegneria e Scienze dell'Informazione e Matematica\\
	Universit\`{a} dell'Aquila}
    
    {barbara.nelli@univaq.it}

    \
    
{Claudia Pontuale}

{Dipartimento di Ingegneria e Scienze dell'Informazione e Matematica\\
	Universit\`{a} dell'Aquila}
    
    {claudia.pontuale@graduate.univaq.it}

\end{document}